\documentclass[11pt]{amsart}
\usepackage[left=1in, right=1in, top=1in]{geometry}
\usepackage[toc,page]{appendix}
\usepackage[utf8]{inputenc}
\usepackage{amsthm}
\usepackage{amsmath}
\usepackage{amsfonts}
\usepackage{tabularray}
\usepackage{hyperref}

\usepackage{tikz}
\usetikzlibrary{matrix, calc}

\UseTblrLibrary{amsmath}

\newtheorem{theorem}{Theorem}[section]
\newtheorem{lemma}[theorem]{Lemma}
\newtheorem*{AL}{Augmentation Lemma}

\title{On Generalized Pfaffians}

\begin{document}

\author{Jacques Distler}
\address{University of Texas at Austin, Austin, TX 78712, USA}
\email{distler@golem.ph.utexas.edu}
\author{Nathan Donagi}
\address{Lower Merion High School, Ardmore, PA 19003, USA}
\email{ndonagi@gmail.com}
\author{Ron Donagi}
\address{University of Pennsylvania, Philadelphia, PA 19104, USA}
\email{Donagi@penn.edu}

\date{\today}

\begin{abstract}
The determinant of a $2N \times 2N$ anti-symmetric matrix $g$ is the square of its Pfaffian, which like the determinant is a polynomial in the entries of $g$. Studies of certain super conformal field theories (`of class S’) suggested a conjectural generalization of this, predicting that each of a series of other polynomials in the entries of $g$ also admit polynomial square roots. Among other consequences, this conjecture led to a characterization of the local Hitchin image for type D. Several important special cases had been established previously. In this paper we prove the conjecture in full.
\end{abstract}

\rightline{UT-WI-15-2024}
\maketitle

\thispagestyle{empty}
\newpage
 
\section{Introduction}

As is well-known, the determinant of an $2N \times 2N$ anti-symmetric matrix $g$ is the square of its Pfaffian, which like the determinant is a polynomial in the entries of $g$. This was conjecturally generalized in [2], Conjecture 1. The starting point is the choice of a nilpotent anti-symmetric matrix $X\in \mathfrak{so}(2N)$ and consideration of the expression

\[det(sX + tg + \lambda I)\]

\noindent which is a polynomial in the entries of $g\in\mathfrak{so}(2N)$ and the variables $s, t, \lambda$. The determinant $\det(g)$ occurs as the coefficient of $t^{2N}$,
and equals the square of the Pfaffian. The conjecture is that certain other coefficients also have polynomial square roots. In more detail:

A D-partition is a partition where all even parts appear with even multiplicity. Nilpotent elements of $\mathfrak{so}(2N)$, up to conjugacy, are labeled by D-partitions of $2N$ (with the proviso that, when $N$ is even, a partition consisting of only even parts corresponds to two \emph{distinct} nilpotent orbits).

A special D-partition is a partition where there is an even number of odd parts between any 2 even parts or at the beginning or end of the partition.

In this paper we will consider all special D-partitions to be ordered from greatest to least: 

\[p_1 \geq p_2 \geq \dots \geq p_{2m}.\]

\noindent We follow the convention for matrices in $\mathfrak{so}(2N)$ from \cite{CollingwoodMcGovern}. A matrix is in $\mathfrak{so}(2N)$ if it is of the form:

\[\begin{bmatrix}A& B\\ C& -A^t\end{bmatrix}\]

\noindent where A is an $N \times N$ matrix and B and C are anti-symmetric. This form is conjugate (over $\mathbb{C}$) to an anti-symmetric matrix and therefore the determinant can be calculated in either form.

Given a special D-partition, $P$, with a nilpotent $X\in \mathfrak{so}(2N)$ in the corresponding nilpotent orbit and a matrix $g\in\mathfrak{so}(2N)$, define $c_{2k_i}$ such that:

\[charpol(x+tg) := det(X+tg+\lambda I) = \lambda^{2N}+\sum_{k=1}^{N} t^{\chi_{2k}}(c_{2k}+ O(t))\lambda^{2(N-k)}.\]

\noindent Here $\chi_{2k}$ is the lowest power of $t$ that multiplies $\lambda^{2(N-k)}$. It is given by the following recipe. For each $j$, define the partial sum
\[
   k_j =  \sum_{i=1}^j p_i.
\]
Then $\chi_{k}=j$, where $j$ is given by demanding
\[
k_{j-1}<k\leq k_{j}.
\]
Note that the condition that $P$ be special ensures that $k_j$ is even when $j$ is even.

The conjecture then states that for every \emph{even} $j$ such that $p_{j}>p_{j+1}$,  $c_{k_{j}}$ is the square of a polynomial on the Lie algebra.

A case where this conjecture is known is when $j=2m$, in which case $c_{2N}$ is the square of the Pfaffian.

This Conjecture was used in \cite{Balasubramanian:2023iyx} in several different ways. The subject of that paper was the local image of the Hitchin map in type-D. Given a nilpotent, $X$, this image was a certain subspace of the $c_{2k}$'s, characterized by a certain set of ``constraints.'' The constraints came in two types, \emph{even-type constraints} (Theorem 2 of \cite{Balasubramanian:2023iyx}) and \emph{odd-type constraints} (Theorem 3 of \cite{Balasubramanian:2023iyx}). In each case, the characterization of the local Hitchin image relied on the fact that certain of the $c_{2k}$'s could be written as squares of polynomial functions on the Lie algebra. The authors of \cite{Balasubramanian:2023iyx} were able to prove certain special cases of the Conjecture (see Proposition 1 of \cite{Balasubramanian:2023iyx}), but the general case eluded them. We fill that gap here.

\smallskip

\section{Proof}

To begin, we will homogenize the characteristic polynomial, by multiplying $X$ by a scalar $s$. Though we are only really interested in the case where $s=1$, this will make the upcoming logic clearer.

After homogenization, the determinant \[det(sX+tg+\lambda I)\] will be a homogeneous polynomial of degree $2N$.

Now for a specific $j$ and $2k$, our $c_{2k}$ arises as the coefficient of \[s^{2k-j} t^j \lambda^{2(N-k)}.\]

\noindent A prescription for generating the nilpotent $X$ from a partition $P$ was given in \cite{CollingwoodMcGovern} (see also \cite{Chacaltana:2011ze}).

Here we give a simplified version of said prescription, ignoring numerical coefficients.

Take a D-partition $[p_1,p_2...p_{2m}]$ and divide it into pairs of the form $[r,r]$ and $[2a+1, 2b+1] (a>b)$. label the pairs $q_1$, $q_2$, \dots, $q_m$.

We take $X$ to be in block diagonal form, consisting of blocks $X^\alpha$, $\alpha = 1,\dots,m$ The $\alpha$th block depends on the type of $q_\alpha$. 

First let $E_{i,j}$ be the $2n \times 2n$ matrix (where 2n is the size of the current block) with a $1$ in the $(i, j)th$ position and zeroes everywhere else.

Then let

\[X^-_{i,j}=E_{i,j}-E_{j+N,i+N}\]

\noindent and

\[X^+_{i,j}=E_{i,j+N}-E_{j,i+N}\] for $i<j$

If $q_\alpha$ is a pair $(r,r)$ then the $\alpha$th block is

\[\sum_{k=1}^{r-1}X^-_{k,k+1}.\]

\noindent We define the positive minor of this block to be its upper left $r \times r$ corner, and its negative minor to be its lower right $r \times r$ corner.

Otherwise, if the $\alpha$th block is of the form $(2a+1, 2b+1)$ the full formula is in \cite{CollingwoodMcGovern}, however we only need the case where $b=0$ in which case the formula is 

\[\sum_{k=1}^{s}X^-_{k,k+1}+X^+_{s,s+1}.\]

As an example, consider the partition $[4,4,3,1]$. First split it into $2$ pairs, $(4,4), (3,1)$. The pair $(4,4)$ yields the minors:

\[
\begin{bmatrix}
0 & 1 & 0 & 0\\
0 & 0 & 1 & 0\\
0 & 0 & 0 & 1\\
0 & 0 & 0 & 0\\
\end{bmatrix}
\]

\noindent and

\[
\begin{bmatrix}
0 & 0 & 0 & 0\\
-1 & 0 & 0 & 0\\
0 & -1 & 0 & 0\\
0 & 0 & -1 & 0\\
\end{bmatrix}
\]

\noindent which combine to the matrix:

\[
\begin{bmatrix}
0 & 1 & 0 & 0 & 0 & 0 & 0 & 0\\
0 & 0 & 1 & 0 & 0 & 0 & 0 & 0\\
0 & 0 & 0 & 1 & 0 & 0 & 0 & 0\\
0 & 0 & 0 & 0 & 0 & 0 & 0 & 0\\
0 & 0 & 0 & 0 & 0 & 0 & 0 & 0\\
0 & 0 & 0 & 0 & -1 & 0 & 0 & 0\\
0 & 0 & 0 & 0 & 0 & -1 & 0 & 0\\
0 & 0 & 0 & 0 & 0 & 0 & -1 & 0\\
\end{bmatrix}
\]

\noindent and the pair $(3,1)$ yields only a single minor/matrix:

\[
\begin{bmatrix}
0 & 1 & 0 & 1\\
0 & 0 & -1 & 0\\
0 & 0 & 0 & 0\\
0 & 0 & -1 & 0\\
\end{bmatrix}
\]

\noindent The nilpotent $X$ is then given by the combined matrix:

\[
\begin{+bmatrix}
0 & 1 & 0 & 0 & 0 & 0 & 0 & 0 & 0 & 0 & 0 & 0\\
0 & 0 & 1 & 0 & 0 & 0 & 0 & 0 & 0 & 0 & 0 & 0\\
0 & 0 & 0 & 1 & 0 & 0 & 0 & 0 & 0 & 0 & 0 & 0\\
0 & 0 & 0 & 0 & 0 & 0 & 0 & 0 & 0 & 0 & 0 & 0\\
0 & 0 & 0 & 0 & 0 & 1 & 0 & 0 & 0 & 0 & 0 & 1\\
0 & 0 & 0 & 0 & 0 & 0 & 0 & 0 & 0 & 0 & -1 & 0\\
0 & 0 & 0 & 0 & 0 & 0 & 0 & 0 & 0 & 0 & 0 & 0\\
0 & 0 & 0 & 0 & 0 & 0 & -1 & 0 & 0 & 0 & 0 & 0\\
0 & 0 & 0 & 0 & 0 & 0 & 0 & -1 & 0 & 0 & 0 & 0\\
0 & 0 & 0 & 0 & 0 & 0 & 0 & 0 & -1 & 0 & 0 & 0\\
0 & 0 & 0 & 0 & 0 & 0 & 0 & 0 & 0 & 0 & 0 & 0\\
0 & 0 & 0 & 0 & 0 & 0 & 0 & 0 & 0 & 0 & -1 & 0
\end{+bmatrix}
\]

\smallskip

In the case where all parts of $P$ are even, there are actually 2 separate orbits, one follows the above prescription, one does not. However the conjecture was already proved for this special case in \cite{Balasubramanian:2023iyx} and therefore the other orbit will be ignored in this paper.

\smallskip

A partition $P'=[p'_1, p'_2, \dots, p'_{m'}]$ is called an augmentation of $P=[p_1, p_2, \dots, p_{m}]$ if $p'_i=p_i$ for $i=1,\dots,m$

\smallskip

\begin{AL}
let $P'$ be an augmentation of $P$, then the conjecture for $P'$ implies the conjecture for $P$.
\end{AL}
\begin{proof}
The Lie algebra $\mathfrak{so}(2N)$ is a sub-algebra of $\mathfrak{so}(2N')$. The determinant for $\mathfrak{so}(2N)$ is the restriction of the determinant for $\mathfrak{so}(2N')$. Clearly, the restriction of a square is a square. This shows that all $c_{2k}$, $k<m$ are still squares over the Lie algebra. In the case where $k=m$ it's already known $c_{2m}$ is the square of the Pfaffian.
\end{proof}

We introduce a new structure we call a “Lego set”, which is comprised of “Lego blocks”.

A Lego block comes in two types:

A Type I Lego block consists of a single number $r$ and a sign $\pm$. These must be added in $\pm$ pairs of the same $r$. The associated minor is the minor with matching sign of the partition that generated the Lego block.

A Type II Lego block consists of a pair of numbers $(2r-1,1)$. The associated minor is the matrix of the pair that generated the block.

A Lego set is sorted by the highest element in each Lego block. From any partition we can generate a Lego set by adding extra pairs $(1,1)$ and using the augmentation lemma.

To convert a partition to a Lego set, first take each $(2r,2r)$ pair and convert it into two Type I Lego blocks of the same $2r$ and opposite signs. For each $(2a+1, 2b+1)$ pair, first add a $(1,1)$ pair to the end of the partition, and then add two type II Lego blocks to the Lego set, ($2a+1,1)$ and $(2b+1, 1)$.

For an example, consider the partition $[7,5,2,2]$, according to the prescription above it gets broken into $(7,5)$ and $(2,2)$. Then from it we get the Lego set $(7,1), (5,1), (2+), (2-)$.

Now we introduce the notion of eliminating a row and column. This is whenever we select a specific non zero entry from $X$ to contribute a power $s$ to the term we look for. Therefore that specific $s$, and its specific contribution to the determinant, will never involve any contributions from other entries in the same row or column, and therefore we say that row and that column are eliminated.

We say a Lego block is 'tapped' if at least one occurrence of $s$ in a monomial of the determinant comes from an $s$ inside the associated minor.

\setcounter{theorem}{0}
\begin{lemma}
\label{oneT}
When a Lego block is tapped, every monomial that occurs must have at least one $t$. 
\end{lemma}

\begin{proof}
For type I, since there is only one Jordan block, the claim is clear. For type II, there is one pair of $s$ that share a row and one that share a column. From each pair at most one $s$ can be chosen, leaving 4 possibilities. In each of the possibilities it is clear that no symmetric minor has full rank.
\end{proof}

\begin{lemma}
There is a unique optimal way to select the rows and columns that are eliminated from a Lego block. (This means that only one row and one column from the associated minor contributes a $t$, the smallest possible number of such.)
\end{lemma}

\begin{proof}
For a type I Lego block, no two $s$ share a row or column, therefore you can choose every $s$ within it and there is clearly only one way to choose every  $s$.

For a type II Lego block, out of the 4 possible ways to choose $2r-2$ $s$'s, only two ways will yield a remaining $\lambda$. (When exactly 1 out of the 2 possible $s$ are chosen from the upper right $r \times r$ minor). However both ways still result in the same rows and columns being eliminated and will therefore still result in the same minor being left at the end. Only the first column and $r+1$-th row will not be eliminated (as the 2 diagonal elements that share either its row or column are eliminated by $s$ that are always chosen, and there are no $s$ in the first column or $r+1$th row).

For an example, in the $(3,1)$, $r=1$ case, either:

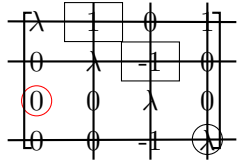
\tikzset{
    ncbar angle/.initial=90,
    ncbar/.style={
        to path=(\tikztostart)
        -- ($(\tikztostart)!#1!\pgfkeysvalueof{/tikz/ncbar angle}:(\tikztotarget)$)
        -- ($(\tikztotarget)!($(\tikztostart)!#1!\pgfkeysvalueof{/tikz/ncbar angle}:(\tikztotarget)$)!\pgfkeysvalueof{/tikz/ncbar angle}:(\tikztostart)$)
        -- (\tikztotarget)
    },
    ncbar/.default=0.5cm,
}

\tikzset{square left brace/.style={ncbar=0.1cm}}
\tikzset{square right brace/.style={ncbar=-0.1cm}}

\[
\begin{tikzpicture}
    \matrix (m) [%
      matrix of nodes,
      text width=5mm,
      text badly centered
    ] {%
      $\lambda$ & 1 & 0 & 1\\
      0 & $\lambda$ & -1 & 0\\
      0 & 0 & $\lambda$ & 0\\
      0 & 0 & -1 & $\lambda$\\
    };
    \draw[thick] (m-1-2.north) -- (m-4-2.south);
    \draw[thick] (m-1-3.north) -- (m-4-3.south);
    \draw[thick] (m-1-4.north) -- (m-4-4.south);
    \draw[thick] (m-1-1.west) -- (m-1-4.east);
    \draw[thick] (m-2-1.west) -- (m-2-4.east);
    \draw[thick] (m-4-1.west) -- (m-4-4.east);
    \draw (m-1-2.south west) rectangle (m-1-2.north east);
    \draw (m-2-3.south west) rectangle (m-2-3.north east);
    \draw (m-4-4) ellipse (2mm and 2mm);
    \draw [red] (m-3-1) ellipse (2mm and 2mm);
    \draw [thick] (-1.2,0.9) to [square right brace ] (-1.2,-0.9);
    \draw [thick] (1.2,0.9) to [square left brace ] (1.2,-0.9);
  \end{tikzpicture}
\]

\noindent or

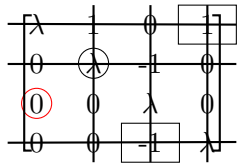
\[
\begin{tikzpicture}
    \matrix (m) [%
      matrix of nodes,
      text width=5mm,
      text badly centered
    ] {%
      $\lambda$ & 1 & 0 & 1\\
      0 & $\lambda$ & -1 & 0\\
      0 & 0 & $\lambda$ & 0\\
      0 & 0 & -1 & $\lambda$\\
    };
    \draw[thick] (m-1-2.north) -- (m-4-2.south);
    \draw[thick] (m-1-3.north) -- (m-4-3.south);
    \draw[thick] (m-1-4.north) -- (m-4-4.south);
    \draw[thick] (m-1-1.west) -- (m-1-4.east);
    \draw[thick] (m-2-1.west) -- (m-2-4.east);
    \draw[thick] (m-4-1.west) -- (m-4-4.east);
    \draw (m-1-4.south west) rectangle (m-1-4.north east);
    \draw (m-4-3.south west) rectangle (m-4-3.north east);
    \draw (m-2-2) ellipse (2mm and 2mm);
    \draw [red] (m-3-1) ellipse (2mm and 2mm);
    \draw [thick] (-1.2,0.9) to [square right brace ] (-1.2,-0.9);
    \draw [thick] (1.2,0.9) to [square left brace ] (1.2,-0.9);
  \end{tikzpicture}
\]

\noindent and then for the full partition [5,3,1,1], with Lego set (5,1),(3,1) the four possibilities are:

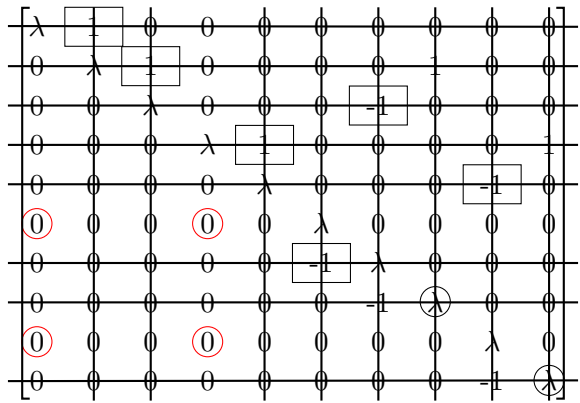
\[
\begin{tikzpicture}
    \matrix (m) [
      matrix of nodes,
      text width=5mm,
      text badly centered
    ] {
      $\lambda$ & 1 & 0 & 0 & 0 & 0 & 0 & 0 & 0 & 0\\
      0 & $\lambda$ & 1 & 0 & 0 & 0 & 0 & 1 & 0 & 0\\
      0 & 0 & $\lambda$ & 0 & 0 & 0 & -1 & 0 & 0 & 0\\
      0 & 0 & 0 & $\lambda$ & 1 & 0 & 0 & 0 & 0 & 1\\
      0 & 0 & 0 & 0 & $\lambda$ & 0 & 0 & 0 & -1 & 0\\
      0 & 0 & 0 & 0 & 0 & $\lambda$ & 0 & 0 & 0 & 0\\
      0 & 0 & 0 & 0 & 0 & -1 & $\lambda$ & 0 & 0 & 0\\
      0 & 0 & 0 & 0 & 0 & 0 & -1 & $\lambda$ & 0 & 0\\
      0 & 0 & 0 & 0 & 0 & 0 & 0 & 0 & $\lambda$ & 0\\
      0 & 0 & 0 & 0 & 0 & 0 & 0 & 0 & -1 & $\lambda$\\
    };

    \draw[thick] (m-1-2.north) -- (m-10-2.south);
    \draw[thick] (m-1-3.north) -- (m-10-3.south);
    \draw[thick] (m-1-5.north) -- (m-10-5.south);
    \draw[thick] (m-1-6.north) -- (m-10-6.south);
    \draw[thick] (m-1-7.north) -- (m-10-7.south);
    \draw[thick] (m-1-8.north) -- (m-10-8.south);
    \draw[thick] (m-1-9.north) -- (m-10-9.south);
    \draw[thick] (m-1-10.north) -- (m-10-10.south);
    
    \draw[thick] (m-1-1.west) -- (m-1-10.east);
    \draw[thick] (m-2-1.west) -- (m-2-10.east);
    \draw[thick] (m-3-1.west) -- (m-3-10.east);
    \draw[thick] (m-4-1.west) -- (m-4-10.east);
    \draw[thick] (m-5-1.west) -- (m-5-10.east);
    \draw[thick] (m-7-1.west) -- (m-7-10.east);
    \draw[thick] (m-8-1.west) -- (m-8-10.east);
    \draw[thick] (m-10-1.west) -- (m-10-10.east);
    
    \draw (m-1-2.south west) rectangle (m-1-2.north east);
    \draw (m-2-3.south west) rectangle (m-2-3.north east);
    \draw (m-3-7.south west) rectangle (m-3-7.north east);
    \draw (m-4-5.south west) rectangle (m-4-5.north east);
    \draw (m-5-9.south west) rectangle (m-5-9.north east);
    \draw (m-7-6.south west) rectangle (m-7-6.north east);

    \draw (m-8-8) ellipse (2mm and 2mm);
    \draw (m-10-10) ellipse (2mm and 2mm);

    \draw [red] (m-6-1) ellipse (2mm and 2mm);
    \draw [red] (m-9-1) ellipse (2mm and 2mm);
    \draw [red] (m-6-4) ellipse (2mm and 2mm);
    \draw [red] (m-9-4) ellipse (2mm and 2mm);

    \draw [thick] (-3.5,2.6) to [square right brace ] (-3.5,-2.6);
    \draw [thick] (3.5,2.6) to [square left brace ] (3.5, -2.6);
  \end{tikzpicture}
\]

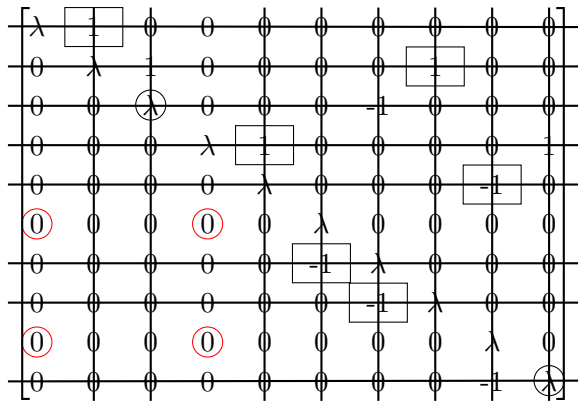
\[
\begin{tikzpicture}
    \matrix (m) [
      matrix of nodes,
      text width=5mm,
      text badly centered
    ] {
      $\lambda$ & 1 & 0 & 0 & 0 & 0 & 0 & 0 & 0 & 0\\
      0 & $\lambda$ & 1 & 0 & 0 & 0 & 0 & 1 & 0 & 0\\
      0 & 0 & $\lambda$ & 0 & 0 & 0 & -1 & 0 & 0 & 0\\
      0 & 0 & 0 & $\lambda$ & 1 & 0 & 0 & 0 & 0 & 1\\
      0 & 0 & 0 & 0 & $\lambda$ & 0 & 0 & 0 & -1 & 0\\
      0 & 0 & 0 & 0 & 0 & $\lambda$ & 0 & 0 & 0 & 0\\
      0 & 0 & 0 & 0 & 0 & -1 & $\lambda$ & 0 & 0 & 0\\
      0 & 0 & 0 & 0 & 0 & 0 & -1 & $\lambda$ & 0 & 0\\
      0 & 0 & 0 & 0 & 0 & 0 & 0 & 0 & $\lambda$ & 0\\
      0 & 0 & 0 & 0 & 0 & 0 & 0 & 0 & -1 & $\lambda$\\
    };

    \draw[thick] (m-1-2.north) -- (m-10-2.south);
    \draw[thick] (m-1-3.north) -- (m-10-3.south);
    \draw[thick] (m-1-5.north) -- (m-10-5.south);
    \draw[thick] (m-1-6.north) -- (m-10-6.south);
    \draw[thick] (m-1-7.north) -- (m-10-7.south);
    \draw[thick] (m-1-8.north) -- (m-10-8.south);
    \draw[thick] (m-1-9.north) -- (m-10-9.south);
    \draw[thick] (m-1-10.north) -- (m-10-10.south);
    
    \draw[thick] (m-1-1.west) -- (m-1-10.east);
    \draw[thick] (m-2-1.west) -- (m-2-10.east);
    \draw[thick] (m-3-1.west) -- (m-3-10.east);
    \draw[thick] (m-4-1.west) -- (m-4-10.east);
    \draw[thick] (m-5-1.west) -- (m-5-10.east);
    \draw[thick] (m-7-1.west) -- (m-7-10.east);
    \draw[thick] (m-8-1.west) -- (m-8-10.east);
    \draw[thick] (m-10-1.west) -- (m-10-10.east);
    
    \draw (m-1-2.south west) rectangle (m-1-2.north east);
    \draw (m-2-8.south west) rectangle (m-2-8.north east);
    \draw (m-8-7.south west) rectangle (m-8-7.north east);
    \draw (m-4-5.south west) rectangle (m-4-5.north east);
    \draw (m-5-9.south west) rectangle (m-5-9.north east);
    \draw (m-7-6.south west) rectangle (m-7-6.north east);

    \draw (m-3-3) ellipse (2mm and 2mm);
    \draw (m-10-10) ellipse (2mm and 2mm);

    \draw [red] (m-6-1) ellipse (2mm and 2mm);
    \draw [red] (m-9-1) ellipse (2mm and 2mm);
    \draw [red] (m-6-4) ellipse (2mm and 2mm);
    \draw [red] (m-9-4) ellipse (2mm and 2mm);

    \draw [thick] (-3.5,2.6) to [square right brace ] (-3.5,-2.6);
    \draw [thick] (3.5,2.6) to [square left brace ] (3.5, -2.6);

  \end{tikzpicture}
\]

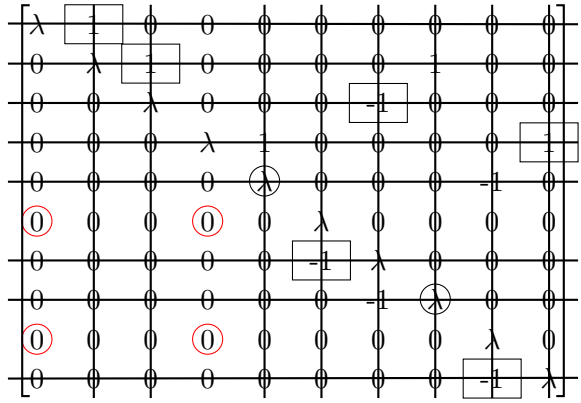
\[
\begin{tikzpicture}
    \matrix (m) [
      matrix of nodes,
      text width=5mm,
      text badly centered
    ] {
      $\lambda$ & 1 & 0 & 0 & 0 & 0 & 0 & 0 & 0 & 0\\
      0 & $\lambda$ & 1 & 0 & 0 & 0 & 0 & 1 & 0 & 0\\
      0 & 0 & $\lambda$ & 0 & 0 & 0 & -1 & 0 & 0 & 0\\
      0 & 0 & 0 & $\lambda$ & 1 & 0 & 0 & 0 & 0 & 1\\
      0 & 0 & 0 & 0 & $\lambda$ & 0 & 0 & 0 & -1 & 0\\
      0 & 0 & 0 & 0 & 0 & $\lambda$ & 0 & 0 & 0 & 0\\
      0 & 0 & 0 & 0 & 0 & -1 & $\lambda$ & 0 & 0 & 0\\
      0 & 0 & 0 & 0 & 0 & 0 & -1 & $\lambda$ & 0 & 0\\
      0 & 0 & 0 & 0 & 0 & 0 & 0 & 0 & $\lambda$ & 0\\
      0 & 0 & 0 & 0 & 0 & 0 & 0 & 0 & -1 & $\lambda$\\
    };

    \draw[thick] (m-1-2.north) -- (m-10-2.south);
    \draw[thick] (m-1-3.north) -- (m-10-3.south);
    \draw[thick] (m-1-5.north) -- (m-10-5.south);
    \draw[thick] (m-1-6.north) -- (m-10-6.south);
    \draw[thick] (m-1-7.north) -- (m-10-7.south);
    \draw[thick] (m-1-8.north) -- (m-10-8.south);
    \draw[thick] (m-1-9.north) -- (m-10-9.south);
    \draw[thick] (m-1-10.north) -- (m-10-10.south);
    
    \draw[thick] (m-1-1.west) -- (m-1-10.east);
    \draw[thick] (m-2-1.west) -- (m-2-10.east);
    \draw[thick] (m-3-1.west) -- (m-3-10.east);
    \draw[thick] (m-4-1.west) -- (m-4-10.east);
    \draw[thick] (m-5-1.west) -- (m-5-10.east);
    \draw[thick] (m-7-1.west) -- (m-7-10.east);
    \draw[thick] (m-8-1.west) -- (m-8-10.east);
    \draw[thick] (m-10-1.west) -- (m-10-10.east);
    
    \draw (m-1-2.south west) rectangle (m-1-2.north east);
    \draw (m-2-3.south west) rectangle (m-2-3.north east);
    \draw (m-3-7.south west) rectangle (m-3-7.north east);
    \draw (m-4-10.south west) rectangle (m-4-10.north east);
    \draw (m-10-9.south west) rectangle (m-10-9.north east);
    \draw (m-7-6.south west) rectangle (m-7-6.north east);

    \draw (m-8-8) ellipse (2mm and 2mm);
    \draw (m-5-5) ellipse (2mm and 2mm);

    \draw [red] (m-6-1) ellipse (2mm and 2mm);
    \draw [red] (m-9-1) ellipse (2mm and 2mm);
    \draw [red] (m-6-4) ellipse (2mm and 2mm);
    \draw [red] (m-9-4) ellipse (2mm and 2mm);

    \draw [thick] (-3.5,2.6) to [square right brace ] (-3.5,-2.6);
    \draw [thick] (3.5,2.6) to [square left brace ] (3.5, -2.6);
  \end{tikzpicture}
\]

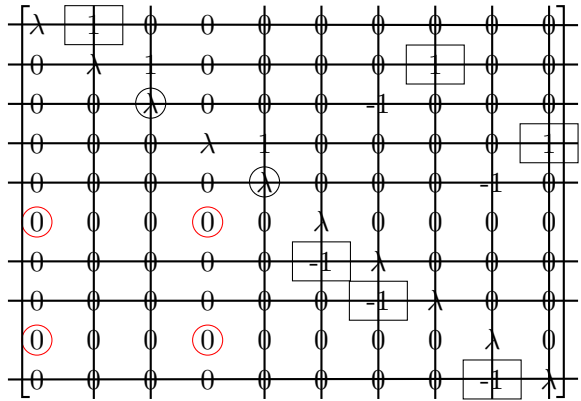
\[
\begin{tikzpicture}
    \matrix (m) [
      matrix of nodes,
      text width=5mm,
      text badly centered
    ] {
      $\lambda$ & 1 & 0 & 0 & 0 & 0 & 0 & 0 & 0 & 0\\
      0 & $\lambda$ & 1 & 0 & 0 & 0 & 0 & 1 & 0 & 0\\
      0 & 0 & $\lambda$ & 0 & 0 & 0 & -1 & 0 & 0 & 0\\
      0 & 0 & 0 & $\lambda$ & 1 & 0 & 0 & 0 & 0 & 1\\
      0 & 0 & 0 & 0 & $\lambda$ & 0 & 0 & 0 & -1 & 0\\
      0 & 0 & 0 & 0 & 0 & $\lambda$ & 0 & 0 & 0 & 0\\
      0 & 0 & 0 & 0 & 0 & -1 & $\lambda$ & 0 & 0 & 0\\
      0 & 0 & 0 & 0 & 0 & 0 & -1 & $\lambda$ & 0 & 0\\
      0 & 0 & 0 & 0 & 0 & 0 & 0 & 0 & $\lambda$ & 0\\
      0 & 0 & 0 & 0 & 0 & 0 & 0 & 0 & -1 & $\lambda$\\
    };

    \draw[thick] (m-1-2.north) -- (m-10-2.south);
    \draw[thick] (m-1-3.north) -- (m-10-3.south);
    \draw[thick] (m-1-5.north) -- (m-10-5.south);
    \draw[thick] (m-1-6.north) -- (m-10-6.south);
    \draw[thick] (m-1-7.north) -- (m-10-7.south);
    \draw[thick] (m-1-8.north) -- (m-10-8.south);
    \draw[thick] (m-1-9.north) -- (m-10-9.south);
    \draw[thick] (m-1-10.north) -- (m-10-10.south);
    
    \draw[thick] (m-1-1.west) -- (m-1-10.east);
    \draw[thick] (m-2-1.west) -- (m-2-10.east);
    \draw[thick] (m-3-1.west) -- (m-3-10.east);
    \draw[thick] (m-4-1.west) -- (m-4-10.east);
    \draw[thick] (m-5-1.west) -- (m-5-10.east);
    \draw[thick] (m-7-1.west) -- (m-7-10.east);
    \draw[thick] (m-8-1.west) -- (m-8-10.east);
    \draw[thick] (m-10-1.west) -- (m-10-10.east);
    
    \draw (m-1-2.south west) rectangle (m-1-2.north east);
    \draw (m-2-8.south west) rectangle (m-2-8.north east);
    \draw (m-8-7.south west) rectangle (m-8-7.north east);
    \draw (m-4-10.south west) rectangle (m-4-10.north east);
    \draw (m-10-9.south west) rectangle (m-10-9.north east);
    \draw (m-7-6.south west) rectangle (m-7-6.north east);

    \draw (m-3-3) ellipse (2mm and 2mm);
    \draw (m-5-5) ellipse (2mm and 2mm);

    \draw [red] (m-6-1) ellipse (2mm and 2mm);
    \draw [red] (m-9-1) ellipse (2mm and 2mm);
    \draw [red] (m-6-4) ellipse (2mm and 2mm);
    \draw [red] (m-9-4) ellipse (2mm and 2mm);

    \draw [thick] (-3.5,2.6) to [square right brace ] (-3.5,-2.6);
    \draw [thick] (3.5,2.6) to [square left brace ] (3.5, -2.6);
  \end{tikzpicture}
\]

\noindent each leaving the same $2 \times 2$ minor.

\end{proof}

Now by Lemma 2, for each tapped Lego block, at least one row/column must contribute a $t$. Therefore, as the degree of $t$ is $j$ in the term we are interested in, if more than $j$ blocks are tapped we end up with a power of $t$ that is too high.

Recall that $2k$ is defined to be the sum of the first $j$ elements of the partition.

Let  $L_\beta$ denote the $\beta$-th element of the Lego set $L$, and let $s_\beta$ be the dimension of the associated minor of $L_\beta$, i.e. $s_\beta = r$ for Type I, and $s_\beta = 2r$ for Type II. 

Let $b_\beta$ be the number of type II Lego blocks among the first $\beta$ blocks. 

Note that there are $2m-b_{2m}$ Lego blocks. However $j>2m-b_{2m}$ only if $j=2m$, in which case the conjecture is already known. For the rest of this proof we assume $j\leq 2m-b_{2m}$.

Now note that $2k = (\sum_{\beta=1}^{j} s_\beta) - b_{j}$. Each Type I has at most $2r-1$ occurrences of $s$ that can be chosen at once, and each Type II has at most $2r-2$ occurrences of $s$ that can be chosen at once. Therefore the sum of the number of occurrences of $s$ in the first $j$ pairs, is:

\[(\sum_{\beta=1}^j s_\beta)-j-b_j = 2k-j.\] 

\noindent (We subtract once the number of type 1 pairs, and twice the number of type 2 pairs). This is the exact number of $s$'s needed. 

The optimal solution to choose such Lego blocks is clearly given by the greedy algorithm, where the largest such blocks are selected one at a time.
This is because replacing any block with a smaller block will clearly leave fewer $s$'s chosen, and therefore will require an additional block to be tapped. Filling a block only partially will not help, as there would be leftovers which must go into a separate block.

\begin{lemma}
The determinant of the final minor left, after eliminating the maximum possible number of rows and columns from the associated minors of the first $j$ Lego blocks, is $\pm$ a square.
\end{lemma}

\begin{proof}
Whenever row $i$ is eliminated, so is column $i\pm n$ and vice versa. Therefore the minor remains in the form:

\[\begin{bmatrix}A& B\\ C& -A^t\end{bmatrix},\]

\noindent though $A$ might now be rectangular. Then note that this form can be converted to an anti-symmetric matrix by swapping columns, and therefore its determinant is $\pm$ a square.

\end{proof}

For each type II Lego block there are 2 ways to select which $s$ are chosen. Both choices result in the same minor, so they introduce one factor of 2 per type II Lego block. As there is always an even number of type II Lego blocks, the full number of ways to select $s$ among all tapped Lego blocks will always be a power of 4.

So letting $z$ be the remaining minor of the matrix $X$, the full coefficient we are looking for is:

\[\pm{2^{b_j}det(z)},\]
with $b_j$ an even integer and $det(z)$ a square, as claimed.
\qed

\section*{Acknowledgements}
During the preparation of this work, RD's research was supported in part by NSF grants DMS 2001673 and 2401422, by NSF FRG grant DMS 2244978, and by Simons HMS Collaboration grant 390287. 
The work of JD was supported in part by the National Science Foundation under Grant No.~PHY--2210562.
\bibliographystyle{utphys}
\bibliography{refs}

\end{document}